\documentclass[11pt,twoside]{article}
\usepackage[english]{babel}

\usepackage{amsthm}
\usepackage{amsmath}
\usepackage{amsfonts}
\usepackage{amssymb}

\usepackage[arrow, matrix, curve]{xy}
\usepackage{picinpar, graphicx}

\usepackage{fancyhdr}
\def\shorttitle{\sc ThŽorie quantique des champs}
\leftmark
\def\shortauthor{\sc\leftmark}

\newtheorem{theorem}{Theorem}[section]
\newtheorem{definition}{Definition}[subsection]

\newtheorem{lemma}{Lemma}[subsection]
\newtheorem{prop}{Proposition}[subsection]
\newtheorem{corollary}{Corollary}[subsection]

\begin{document}

%%\title[Energy-Momentum L.C.L.  and Generalized Isometric Embeddings]{Energy-Momentum's Local Conservation Laws and  Generalized Isometric Embeddings of Vector Bundles}
\title{Cartan-K\"ahler Theory and Applicationsto Local Isometric and Conformal Embedding}

\author{Nabil Kahouadji}

%%\date{August 18, 2008.}
\date{}

%%\address{Institut de Math\'ematiques de Jussieu UMR CNRS 7586, Universit\'e Paris Diderot-Paris VII, case 7012, 2  place jussieu, 75251-Paris Cedex 05, France.}
%%\email{\href{mailto:kahouadji@math.jussieu.fr}{kahouadji@math.jussieu.fr}}

%%\urladdr{\href{http://www.math.jussieu.fr/~kahouadji}
%%         {http://www.math.jussieu.fr/\lower3pt\hbox{\symbol{'176}}kahouadji}}

%\subjclass{
%58A15, %   Exterior differential systems (Cartan theory)
 %37K05, %   Conservation laws
 %32C22 % Embedding of analytic spaces}

%%\keywords{Conservation laws, Generalized isometric embeddings of vector bundles,  Exterior differential systems, Cartan--K\"{a}hler theory , Conservation laws for energy-momentum tensor.}

\maketitle

\thispagestyle{empty}

\begin{abstract}
The goal of this lecture is to give a brief introduction to Cartan-K\"ahler's theory. As examples to the application of this
theory, we choose the local isometric and conformal embedding. We provide lots of details 
and explanations of the calculation and the tools used\footnote{See the Master Thesis \cite{M2Thesis} on which the lecture is based}.
\end{abstract}

\section{Cartan's Structure Equations}

Let $\xi=(E,\pi,M)$ be a vector bundle. Denote $(\mathfrak{X}(M), [, ])$ the Lie algebra of vector fields on $M$ and $\Gamma(E)$ the moduli space of cross-sections of the vector bundle $E$.

\subsection{Connection on a vector bundle} 
A connection on a vector bundle $E$ is a choice of complement of vertical vector fields on $E$. A connection induces a covariant differential operator $\nabla$ on $E$. \linebreak A covariant derivative $\nabla$ on a vector bundle $E$ is a way to ''differentiate'' bundle sections along tangent vectors and it is sometimes called a connection. 

\begin{definition}
A connection  on a vector bundle $E$ is an linear operator defined as follows:
\begin{equation}\nonumber
\begin{split}
\nabla:\mathfrak{X}(M)\times\Gamma(E)&\longrightarrow\Gamma(E)\\
(X,S)&\longmapsto\nabla_{X}S
\end{split}
\end{equation}
\newpage
satisfying

\begin{align}
\nabla_{(X_{1}+X_{2})}S&=\nabla_{X_{1}}S+\nabla_{X_{2}}S , & \nabla_{(fX)}S&=f\nabla_{X}S \notag\\
\nabla_{X}(S_{1}+S_{2})&=\nabla_{X}S_{1}+\nabla_{X}S_{2} , &
\nabla_{X}(fS)&=X(f)S+f\nabla_{X}S \notag
\end{align}

$\forall X, X_{1}, X_{2}, Y, Y_{1}, Y_{2}\in \mathfrak{X}(M)$ and
$\forall S, S_{1}, S_{2}\in \Gamma(E)$.

\end{definition}
\subsubsection{Curvature of a Connection}
\begin{definition}The curvature of a connection $\nabla$ is a vector valued 2-form
\begin{equation}\nonumber
\begin{split}
\mathcal{R}:\mathfrak{X}(M)\times\mathfrak{X}(M)\times\Gamma(E)&\longrightarrow\Gamma(E)\\
X,Y,S&\longmapsto\mathcal{R}(X,Y)S
\end{split}
\end{equation}
defined by $\mathcal{R}(X,Y)S=\Big([\nabla_{X},\ \nabla_{Y}]-\nabla_{[X,\
Y]}\Big)S$
\end{definition}
\begin{theorem}\label{ThCourbureConnexion}
For any $f,g$ and $h$ smooth functions on $M$, $S\in \Gamma(E)$ a section of $\xi$ and
$X,Y \in \mathfrak{X}(M)$ two tangent vector fields of  $M$, we have
\begin{equation}
\mathcal{R}(fX,gY)(hS)=fgh\mathcal{R}(X,Y)S
\end{equation}
\end{theorem}
\subsubsection{Connection and  Curvature Forms } 
Let $\xi=(E,\pi,M)$ be a vector bundle over a smooth manifold $M$ with an \linebreak $r$-dimensional vector space $E$ as a standard fiber. Let $\nabla$ be a connection on  $\xi$ and $\mathcal{R}$ its
curvature.  We deonte by $\mathcal{U}$ an open set of $M$.

\begin{definition}
A set of $r$ local sections $S=(S_{1},S_{2},\dots , S_{r})$ of $\xi$ is called a frame field (or a moving frame) if $\forall p \in
\mathcal{U}$, $S(p)=\Big(S_{1}(p),S_{2}(p),\dots , S_{r}(p)\Big)$
form a basis of the fiber  $E_{p}$ over $p$.
\end{definition}

Let $S=(S_{1},S_{2},\dots , S_{r})$ be a frame field,$\nabla$ a connection
on $\xi$ and \linebreak $X\in \mathfrak{X}(M)$ a tangent vector field on  $M$.
Then $\nabla_{X}S_{j}$ is another section of  $\xi$ and it can be expressed in
the frame field $S$ as follows:
\begin{equation}
\nabla_{X}S_{j}=\sum_{i=1}^{r}\omega_{ij}(X)S_{i}
\end{equation}
where $\omega_{ij}\in \mathcal{A}^{1}(M)$ are
differential 1-forms\footnote{$ \mathcal{A}^{k}(M)$ denote the set of differential $k$-form on $M$ (we choose this notation instead of the standard notation $\Omega^{k}(M)$ to not mix   with the curvature form.}
 on $M$ and $\omega_{ij}(X)$ are  smooth functions on $M$.

\begin{definition}
The $r\times r$ matrix $\omega=(\omega_{ij})$
 is called the  connection 1-form of $\nabla$.
\end{definition}

The connection $\nabla$ is completely determined by the matrix
$\omega=(\omega_{ij})$. Conversely, a matrix of
differential 1-forms on $M$ determines a connection \linebreak (in a non-invariant way depending on the choice of the moving frame).

\paragraph{}Let $X,Y \in \mathfrak{X}(M)$ two tangent vector fields. Then
$\mathcal{R}(X,Y)S_{j}$ are sections of  \linebreak  $\xi$., and can be
expressed on the frame field $S$ as follows:
\begin{equation}
\mathcal{R}(X,Y)S_{j}=\sum_{i=1}^{r}\Omega_{ij}(X,Y)S_{i}
\end{equation}
where $\Omega_{ij}\in \mathcal{A}^{2}(M)$ are
differential 2-forms on $M$ and $\Omega_{ij}(X,Y)$ are smooth functions on $M$.

\begin{definition}
The $r\times r$matrix $\Omega=(\Omega_{ij})$ whoose entries are differential
2-forms, is called the  curvature 2-form of the connection $\nabla$.
\end{definition}

We state the following theorem\footnote{In the tangent bundle case, this theorem gives Cartan's second equation, as we will see later.} that gives the relation
between the  \linebreak connection 1-form $\omega$
 and  the curvature 2-form $\Omega$.
\begin{theorem}\label{ThFormConnexFormCourbure}
\begin{equation}\label{EquaStructCartan}
d\omega+\omega\wedge\omega = \Omega \quad\text{(matrix form)}
\end{equation}
or
\begin{equation}\label{2ndequastructureCartan}
d\omega_{ij}+\sum_{k=1}^{m}\omega_{ik}\wedge\omega_{kj}=\Omega_{ij} \quad\text{(on components)}
\end{equation}
\end{theorem}

\subsection{The Induced Connection} Let  $\xi=(E,\pi, M)$ and $\xi
'=(E',\pi ', M)$ be two vector bundles on $M$.
Consider a map  $f:M\longrightarrow M$ and denote 
$\tilde{f}:E\longrightarrow E'$ the associated bundle map i.e. $(f, \tilde{f})$ satisfies the following commutative diagramme:
$$
\begin{xy}
\xymatrix{
      E \ar[r]^{\tilde{f}} \ar[d]_{\pi}    &   E' \ar[d]^{\pi '}\\
      M \ar[r]_{f}             &   M}
\end{xy}
$$

If $\nabla '$ is a connection on $E'$, the vector bundle morphism induces a pull-back connection on $E$

\begin{equation}
\nabla=\tilde{f}^{\ast}\nabla '
\end{equation}

such that for any $S'\in \Gamma(E')$ and $X\in \mathfrak{X}(M)$, $\nabla_{X}(\tilde{f}^{\ast}S')= \linebreak (\tilde{f}^{\ast}\nabla
')_{X}(\tilde{f}^{\ast}S')= \tilde{f}^{\ast}\Big(\nabla_{f_{\ast}X}S'\Big)$ where $f_{\ast ,p}:T_{p}M\longrightarrow T_{f(p)}M$ is the  \linebreak linear tangent map .\\

We can also induce a connection on $\xi$ by another way. 
The connection $\nabla '$ is completely determined by the matrix of differential 
 $1$-forms $\omega'=(\omega_{ij}')$, and we define $\nabla$  by the matrix 
$\omega$ whose entries $\omega_{ij}$ are the pull-back of $\omega_{ij}'$ by$\tilde{f}$, i.e. $\omega=\tilde{f}^{\ast}\omega '$.\\

The pull back commute with the exterior differentiation and with the \linebreak  exterior product\footnote{$d(f^{\ast}\alpha)=f^{\ast}(d\alpha)$ and
$f^{\ast}(\alpha\wedge\beta)=f^{\ast}(\alpha)\wedge
f^{\ast}(\beta)$ for all $\alpha , \beta \in \mathcal{A}(M)$.},
so, the curvature 2-form  $\Omega$ of $\nabla$  is the 
pull back of the curvature 2-form of  $\nabla '$, i.e.
$\Omega=\tilde{f}^{\ast}\Omega '$.

\subsection{Metric Connection} Let $\xi=(E,\pi , M)$ be a vector bundle. We denote by  $\nabla$ a connection on  \linebreak $\xi$ determined by a matrix of  1-forms $\omega$. Let $\Omega$ be the associated curvature  \linebreak 2-form and $g$ a Riemannian metric on $\xi$ (i.e. a positively-defined scalar product on each fiber).

\begin{definition}
$\nabla$ is a connection on $\xi$ compatible with the metric  $g$ (or a metric connection) if  $\nabla$ satisfies to the following property (Leibniz's identity):
\begin{equation}
\begin{split}
&\nabla_{X}\Big(g(S_{1},S_{2})\Big)=g(\nabla_{X}S_{1},S_{2})+g(S_{1}  
\nabla_{X}S_{2})\\ & \forall S_{1},S_{2}\in \Gamma(E), \text{ and
} \forall X\in \mathfrak{X}(M)
\end{split}
\end{equation}
\end{definition}

\begin{prop}\label{PropMatriceConnexAntiSym}
Let $S=(S_{1},S_{2},\dots , S_{n})$ be an orthonormal frame field with respect to $g$, i.e. $g_{p}(S_{i},S_{j})=\delta_{ij}$ for all
$p\in \mathcal{U}$, $i,j=1,\dots ,r$, then the matrix of  1-forms $\omega$ associated to  $S$  and  the curvature matrix of 2-form  are both skew-symmetric.
\end{prop}

\subsection{Tangent Bundle Case}

\subsubsection{Torsion of a Connexion on a Tangent Bundle}
Let us consider now, a local frame field $S=(S_{1},
\dots , S_{m})$ over $\mathcal{U}\subset M$ where $S_{i}\in \mathfrak{X}(\mathcal{U})$ are local tangent vectors fields (i.e. local sections of the tangent bundle such that for all $p\in \mathcal{U}$,
$\Big(S_{1}(p),\dots , S_{m}(p)\Big)$ forms a basis of the tangent vector space of  $M$).
\begin{definition}
If $S$ is a local orthonormal frame field, the associated  \linebreak coframe field $\eta=\Big(\eta_{1},\dots , \eta_{m}\Big)$ is a local frame field of 1-forms, such that for all $p\in \mathcal{U}$,
$\eta_{i}(p)\Big(S_{j}\Big)=\delta_{ij}.$
\end{definition}
We define then a differential 2-form $\Theta$ as follows:
\begin{equation}\label{1ereequaStructureCartan}
d\eta+\omega\wedge\eta = \Theta
\end{equation}

\begin{definition}
$\Theta$ is called  the torsion 2-form of $\nabla$.
\end{definition}

\begin{prop}\label{PropBianchi} On a tangent bundle, the four forms
$\eta, \omega , \Theta$ and $\Omega$ are connected by the folowing equations
\begin{equation}
d\Theta + \omega\wedge \Theta= \Omega\wedge \eta
\end{equation}
and
\begin{equation}\label{biachi1forme}
d\Omega = \Omega\wedge \omega - \omega\wedge \Omega
\end{equation}
\end{prop}

The equation (\ref{biachi1forme}) is the expression of the Bianchi identity via the \linebreak  connection  1-form and the curvature 2-form. Equation (\ref{biachi1forme}) is also valid on \linebreak  a arbitrary vector bundle.

\subsubsection{Cartan's Structure Equations}

Let $(M,g)$ be  an $m$-dimensional Riemannian manifold. Let $\eta=(\eta_{1},\eta_{2},\dots , \eta_{m})$ an orthonormed coframe field on $M$ ($\eta_{j}\in
\mathcal{A}^{1}(M)$). According to equations 
(\ref{1ereequaStructureCartan}),  \linebreak (\ref{2ndequastructureCartan}) and the proposition \ref{PropMatriceConnexAntiSym}, we establish the Cartan structure equations:

\begin{equation}
\begin{cases}
\displaystyle{d\eta_{i}+\sum_{j=1}^{m}\omega_{ij}\wedge\eta_{j}}=0\quad \text{ (torsion-free) }\\
\displaystyle{d\omega_{ij}+\sum_{k=1}^{m}\omega_{ik}\wedge\omega_{kj}=\Omega_{ij}}
\end{cases}
\end{equation}

where the matrix $(\omega_{ij})$ is the L\'evi-Civita connection 1-form (free \linebreak  torsion connection which is compatible with the riemannian metric $g$).  \linebreak Because $\eta$ is an orthonormed coframe field, $(\omega_{ij})$ is skew-symmetric  \linebreak (proposition \ref{PropMatriceConnexAntiSym}.). $(\Omega_{ij})$ is
the curvature 2-form  matrix of the riemannian connection ($\displaystyle{\Omega_{ij}=\frac{1}{2}\sum_{k,l=1}^{m}\mathcal{R}_{ijkl}\eta_{k}\wedge\eta_{l}}$).

\subsection{The Cartan Lemma}

We end  this section with a technical lemma,
which is  easy to establish and at the same time rich applications . This lemma will not only be useful  \linebreak for isometric embedding problem, but also for many calculus in differential geometry.

\begin{lemma}\label{LemmedeCartan}
Let $M$ be an  $m$-dimensional manifold.
$\omega_{1},\omega_{2},\dots,\omega_{r}$ a set of linearly independent differential  $1-$forms
 ($r\leq n$) and
$\theta_{1},\theta_{2},\dots , \theta_{r}$ differential  $1-$forms
such that
\begin{equation}
\sum_{i=1}^{r}\theta_{i}\wedge\omega_{i}=0
\end{equation}
then there exists $r^{2}$  functions $h_{ij}$ in $\mathcal{C}^{1}(M)$ such that
\begin{equation}
\theta_{i}=\sum_{j=1}^{r}h_{ij}\omega_{j}\quad \text{ with }
h_{ij}=h_{ji}.
\end{equation}
\end{lemma}

\section{Exterior Differential Systems and Ideals}

\subsection{Exterior Differential Systems}

Denote $\mathcal{A}(M)$ the space of smooth differential forms\footnote{This is a graded algebra under the wedge product.} on $M$.

\begin{definition}

An exterior differential system is a finite set of differential
forms $I=\{\omega_{1},\omega_{2},\dots , \omega_{k}\}\subset
\mathcal{A}(M)$ for which there is a set of
equations $\{\omega_{i}=0 | \omega_{i}\in I \}$. \\
such that one can write the exterior differential system as follow:
\begin{equation}\nonumber
\begin{cases}
\omega_{1}=0\\
\omega_{2}=0\\
\vdots\\
\omega_{k}=0
\end{cases}
\end{equation}
\end{definition}

\begin{definition}
An exterior differential system $I \subset \mathcal{A}(M)$ is said to be 
Pffafian if  $I$ contains only differential 1-forms, i.e. $I
\subset\mathcal{A}^{1}(M)$.
\end{definition}

\subsection{Exterior Ideals}
\begin{definition}
Let $\mathcal{I} \subset \mathcal{A}(M)$ a set of differentiable
forms. $\mathcal{I}$ is an exterior ideal if:
\begin{enumerate}
\item The exterior product of any differential form of
$\mathcal{I}$ by a differential form of $\mathcal{A}(M)$ belong to
$\mathcal{I}$. \item The sum of any two differential forms of the
same degree belonging to $\mathcal{I}$, \linebreak belong also to $\mathcal{I}$.
\end{enumerate}
\end{definition}

\begin{definition}
Let $I\subset\mathcal{A}(M)$ an exterior differential system. The
exterior ideal generated by $I$ is the smallest exterior ideal
containing $I$.
\end{definition}

\subsection{Exterior Differential Ideals}
\begin{definition}
Let $\mathcal{I}\subset \mathcal{A}(M)$ a set of differential
forms. $\mathcal{I}$ is an \linebreak  exterior differential ideal if
$\mathcal{I}$ is an exterior ideal closed under the exterior
differentiation, i.e. $\forall \omega \in \mathcal{I}, d\omega \in
\mathcal{I}$ (we can also write $d\mathcal{I}\subset
\mathcal{I}$).
\end{definition}

\begin{definition}
Let $I\subset\mathcal{A}(M)$ an exterior differential system. The
exterior \linebreak  differential ideal generated by $I$ is the smallest
exterior differential ideal \linebreak containing $I$.
\end{definition}

\subsection{Closed Exterior Differential Systems}
\begin{definition}
An exterior differential system $I\subset\mathcal{A}(M)$ is said
closed if the exterior differentiation of any form of $I$, belong to
the exterior ideal generated by $I$.
\end{definition}

\begin{prop}\label{PropIuniondIestferme}
An exterior differential system  $I$ is closed if and only if the
exterior differential ideal generated by $I$ is equal to the
exterior ideal generated by $I$. In particular,  $I\cup dI$ is
closed .
\end{prop}

\subsection{Solutions of an Exterior Differential System}
\begin{definition}
Let $I\subset\mathcal{A}(M)$ be an exterior differential system and
 $N$ a sub-manifold of $M$. $N$ is an integral manifold of  $I$ if  $i^{\ast}\omega=0, \forall \omega\in
I$, where $i$ is an embedding $i:N\longrightarrow M$.
\end{definition}

\section{Introduction to Cartan-K\"ahler Theory}

We consider in this section, an $m$-dimensional real manifold $M$ and $\mathcal{I}\subset \mathcal{A}(M)$\linebreak 
an exterior differential ideal on $M$.

\subsection{Integral Elements}
\begin{definition}Let $z\in M$ and $E\subset T_{z}M$ a linear subspace of
$T_{z}M$. $E$ is an integral element of $\mathcal{I}$ if
$\varphi_{E}=0$ for all $\varphi \in \mathcal{I}$. We denote by
$\mathcal{V}_{p}(\mathcal{I})$ the set of $p$-dimensional integral
elements of $\mathcal{I}$.
\end{definition}

\begin{definition}
$N$ is an integral manifold of $\mathcal{I}$ if and only if each
tangent space of  $N$ is an integral element of $\mathcal{I}$.
\end{definition}

\begin{prop} If $E$ is a $p$-dimensional integral element of $\mathcal{I}$,
 then every subspace of  $E$ are also  integral elements of  $\mathcal{I}$.
\end{prop}

We denote by $\mathcal{I}_{p}=\mathcal{I}\cap\mathcal{A}^{p}(M)$ the set of differential $p$-forms of
$\mathcal{I}$.

\begin{prop}\label{propVp}
$\mathcal{V}_{p}(\mathcal{I})=\{E\in G_{p}(TM) | \varphi_{E}=0
\text{ for all } \varphi \in \mathcal{I}_{p} \}$
\end{prop}

\begin{definition}
Let $E$  an integral element of $\mathcal{I}$. Let
$\{e_{1},e_{2},\dots ,e_{p}\}$ a basis of  $E\subset T_{z}M$. The
polar space of $E$, denoted by $H(E)$, is the vector space defined
as follow:
\begin{equation}
H(E)=\{v\in T_{z}M | \varphi(v, e_{1},e_{2},\dots , e_{p})=0
\text{ for all } \varphi \in \mathcal{I}_{p+1}\}.
\end{equation}
\end{definition}

Notice that $E\subset H(E)$. This implies that a
differential form is alternate. The polar space plays an
important role in exterior differential system theory as we shall
see in the following proposition.

\begin{prop}
Let $E \in \mathcal{V}_{p}(\mathcal{I})$ be an $p$-dimensional
integral element of
 $\mathcal{I}$. \linebreak A $(p+1)$-dimensional vector space $E^{+}\subset T_{z}M$ which contains $E$
 is an integral element of  $\mathcal{I}$ if and only if $E^{+}\subset H(E)$.
\end{prop}

In order to check if a given $p$-dimensional integral element of an
exterior differential ideal $\mathcal{I}$ is contained in a
$(p+1)$-dimensional integral element of  $\mathcal{I}$ , we
introduce the following function $r:\mathcal{V}_{p}(\mathcal{I})  \longrightarrow
\mathbb{Z}$, $r(E)=dim H(E)-(p+1)$ is  a relative integer, $\forall E\in \mathcal{V}_p(\mathcal{I})$. \\

Notice that $r(E)\geq 1$. If $r(E)=-1$, then  $E$ is contained in
any $(p+1)$-dimensional integral element of $\mathcal{I}$.

\subsubsection{K\"ahler-Ordinary and
K\"ahler-Regular Integral Elements}

Let $\Delta$ a differential  $n$-form on a $m$-dimensional manifold $M$. Let\footnote{$G_{n}(TM)$  is the Grassmanian of $TM$, i.e. the set
of $n$-dimensional subspace of $TM$.} \linebreak
$G_{n}(TM,\Delta)=\{E\in G_{n}(TM) / \Delta_{E}\neq
0\}$. We denote by
$\mathcal{V}_{n}(\mathcal{I},\Delta)=\mathcal{V}_{n}(\mathcal{I})\cap
G_{n}(TM,\Delta)$ the set of integral elements of  $\mathcal{I}$
on which  $\Delta_{E}\neq 0$.

\begin{definition}
An integral element  $E\in \mathcal{V}_{n}(\mathcal{I})$ is called
 K\"ahler-ordinary if \linebreak there exists a differential $n$-form
 $\Delta$ such that $\Delta_{E}\neq0$. Moreover, if the function $r$ is locally constant in some neighborhood of  $E$,
then $E$ is said  K\"ahler-regular. \linebreak 
\end{definition}

\subsubsection{Integral Flags, Ordinary and Regular Integral Elements}
\begin{definition}
An integral flag of  $\mathcal{I}$ on $z\in M$ of length $n$ is a
sequence of integral elements $E_{k}$ of $\mathcal{I}$: $(0)_{z}\subset E_{1}\subset E_{2}\subset \dots \subset
E_{n}\subset T_{z}M$.
\end{definition}

\begin{definition}
Let $I$ be an exterior differential system on $M$. An integral element
 $E\in \mathcal{V}(I)$ is said ordinary
if its base point $z\in M$ is an ordinary zero of $I_{0}=I\cap
\mathcal{A}^{0}(M)$ and if there exists an integral flag
$(0)_{z}\subset E_{1}\subset E_{2}\subset \dots \subset E_{n}=E
\subset T_{z}M$ where the $E_{k}$, $k=1,\dots ,(n-1)$ are
K\"ahler-regular integral elements. Moreover, if  $E$ is
K\"ahler-regular, then  $E$ is said  regular.
\end{definition}

\subsection{Cartan's Test}

\begin{theorem}\label{TestCartan}(Cartan's test)\\
Let  $\mathcal{I}\subset \mathcal{A}^{\ast}(M)$ be  an exterior ideal
which does not contain  0-forms (functions on $M$). Let  $(0)_{z}\subset E_{1}\subset
E_{2}\subset \dots \subset E_{n}\subset T_{z}M$ be an integral flag
of $\mathcal{I}$. For any $k<n$, we denote by $c_{k}$ the
codimension of the polar space $H(E_{k})$ in $T_{z}M$. Then
$\mathcal{V}_{n}(\mathcal{I})\subset G_{n}(TM)$ is at least of
$c_{0}+c_{1}+\dots + c_{n-1}$ codimension at $E_{n}$.\\
Moreover, $E_{n}$ is an ordinary integral flag if and only if
$E_{n}$ has a neighborhood $U$ \linebreak  in $G_{n}(TM)$ such that
$\mathcal{V}_{n}(\mathcal{I})\cap U$ is a manifold of
$c_{0}+c_{1}+\dots + c_{n-1}$ codimension in  $U$.
\end{theorem}
\begin{proof}
See \cite{ExtDiffSys-Book}, page 74.
\end{proof}

\begin{prop}\label{PropCpdeHCp} Let $\mathcal{I}\cap \mathcal{A}^{\ast}(M)$
an exterior ideal which do not contains 0-forms. Let $E\in
\mathcal{V}_{n}(\mathcal{I})$ be an integral element of $\mathcal{I}$
at the point $z\in M$. Let $\omega_{1},\omega_{2},\dots ,
\omega_{n},\pi_{1},\pi_{2}, \dots , \pi_{s}$ (where $s=$dim$M -
n$) be a coframe in a open neighborhood of  $z\in M$ such that $E=\{v
\in T_{z}M / \pi_{a}(v)=0 \text{ for all } a=1,\dots s\}$. For
all  $p\leq n$, we define $E_{p}=\{v\in E | \omega_{k}(v)=0 \text{
for all } k > p \}$. Let $\{\varphi_{1}, \varphi_{2},\dots
,\varphi_{r}\}$ be the set of differential forms which generate the
exterior ideal  $\mathcal{I}$, where $\varphi_{\rho}$ is of $(d_{\rho}+1)$ degree.\\
For all $\rho$, there exists an expansion
\begin{equation}
\varphi_{\rho}=\sum_{|J|=d_{\rho}}\pi_{\rho}^{J}\wedge\omega_{J}+\tilde{\varphi}_{\rho}
\end{equation}
where the   1-forms $\pi_{\rho}^{J}$ are linear combinations of
the forms $\pi$ and the terms $\tilde{\varphi}_{\rho}$ are, either
of degree  2 or more on  $\pi$, or vanish at $z$.\\
moreover, we have
\begin{equation}
H(E_{p})=\{v\in T_{z}M | \pi_{\rho}^{J}(v)=0 \text{ for all } \rho
\text{ and } \sup J\leq p\}
\end{equation}
In particular, for the integral flag $(0)_{z}\subset E_{1}\subset
E_{2}\subset\dots \subset E_{n}\cap T_{z}M$ de $\mathcal{I}$, \linebreak 
$c_{p}$ correspond to the number of linear independent  forms \linebreak 
$\{\pi_{\rho}^{J}|_{z} \text{ such that } \sup J \leq p\}$ .
\end{prop}

\begin{proof}
See \cite{ExtDiffSys-Book}, page 80.
\end{proof}

\subsection{Cartan-K\"ahler's Theorem}
\paragraph{}The following theorem is a  generalization of the well-known Frobenius's  \linebreak  theorem.

\begin{theorem}\label{ThCartanKahler}(Cartan-K\"ahler)\\
Let $\mathcal{I}\subset \mathcal{A}^{\ast}(M)$ be a real analytic
exterior differential ideal. Let $P\subset M$ a $p$-dimensional connected
real analytic K\"ahler-Regular integral manifold of $\mathcal{I}$.\\
Suppose that $r=r(P)\geq 0$. Let  $R\subset M$ be a real analytic
submanifold of $M$ of  codimension $r$ which contains  $P$ and such
that  $T_{x}R$ and
$H(T_{x}P)$ are transversals in  $T_{x}M$ for all  $x\in P \subset M$.\\
There exists a $(p+1)$-dimensional connected real
analytic integral manifold $X$ of $\mathcal{I}$, such that
$P\subset X \subset R$. $X$
is unique in the sense that another  \linebreak integral manifold of
$\mathcal{I}$ having the stated properties, coincides
with $X$ on a open  \linebreak neighborhood of  $P$.
\end{theorem}

\begin{proof}
See \cite{ExtDiffSys-Book}, page 82.
\end{proof}

The analicity condition of the exterior differential ideal is crucial because of the requirements in the Cauchy-Kowalewski theorem used in the proof of the
Cartan-K\"ahler theorem.\\

Cartan-K\"ahler's theorem has an important corollary. Actually,
this  \linebreak corollary is often more used than the theorem
and it is sometimes called  \textit{the Cartan-K\"ahler theorem}.

\begin{corollary}\label{CorCartanKahler}(Cartan-K\"ahler)\\
Let $\mathcal{I}$ be an analytic exterior differential ideal on a manifold
$M$. If $E\subset T_{z}M$ is  \linebreak an ordinary integral element of
$\mathcal{I}$, there exists an integral manifold of
$\mathcal{I}$ passing through $z$ and having $E$ as a tangent
space at the point $z$.
\end{corollary}

\section{Local Isometric Embedding Problem}
\paragraph{}We shall state and prove the Burstin-Cartan-Janet-Schlaefli's theorem \linebreak  concerning
local isometric embedding of a real analytic Riemannian manifold.
The names of the mathematicians are given in alphabetic order.
Schlaefli in his paper in 1871 \cite{Schlaefli-Art} conjectured
that an $m$-dimensional Riemannian manifold can always be, locally,  embedded
in an $N=\frac{1}{2}m(m+1)$ dimensional Euclidean space. In 1926,
Janet \cite{Janet-Art} proved the result for the dimension 2 by
resolving a differential system and explain how we get the result in the general case.
 In 1927, \'Elie Cartan \cite{Cartan-Art} gave the
complete proof of the result. His method is based on his theory of
involutive Pfaffian system. Later in 1931, Burstin
\cite{Burstin-Art} generalized Janet's method and obtained the
result in the general case.

The proof that we shall give is inspired by Cartan'paper
\cite{Cartan-Art}, the  Bryant, Chern, Gardner, Goldscmidt et
Griffiths's book \cite{ExtDiffSys-Book} and the Griffiths et
Jensen's book \cite{GriffithJensen-Book}.

\subsection{The Burstin-Cartan-Janet-Schlaefli theorem}
\begin{theorem}\label{ThBCJS}(Burstin 1931-Cartan 1927-Janet 1926-Schlaefli 1871)\\
Every $m$-dimensional real analytic Riemannian manifold can be
locally  \linebreak embedded isometrically in an
$\displaystyle{\frac{m(m+1)}{2}}$-dimensional Euclidean space.
\end{theorem}

\subsection{Proof of
Burstin-Cartan-Janet-Schlaefli's theorem}

\paragraph{\textbf{Steps of the proof of theorem\ref{ThBCJS}.}}

\begin{enumerate}
\item We shall write down the Cartan structure equations for an $m$-dimensional
real analytic Riemannian manifold $M$.

\item We shall define a subbundle $\mathcal{F}_{m}(\mathbb{E}^{N})$ of the   bundle
$\mathcal{F}(\mathbb{E}^{N})$ of the \linebreak 
Euclidean space $\mathbb{E}^{N}$, then shall write down the Cartan structure \linebreak 
equations for the subbundle $\mathcal{F}_{m}(\mathbb{E}^{N})$.

\item Given an exterior differential system  $I_{0}$ on
$M\times\mathcal{F}_{m}(\mathbb{E}^{N})$, which is not close, we
shall prove Claim \ref{ProCartanJanet}, which proves that the existence of a local isometric embedding  of $M$ is
the existence of an $m$-dimensional integral manifold of $I_{0}$.

\item We will extend  this differential system to obtain  a closed one. In the process of extension, we will get new
equations  (the Gauss equation (equ.  \linebreak \ref{EquaGauss})). We will also
show that a closed exterior differential system $\tilde{I}$ with fewer 
 1-forms than  $I$, will generate the same
differential ideal that the one generated by $I$ if the  Gauss's
equation is satisfied.

\item We establish the lemma \ref{LemmeSubSurjectThm}., that ensure that the Gauss equations is a surjective submersion.
We shall obtain a submanifold with a known
dimension.

\item Given the closed exterior differential ideal,, we shall prove
the existence of an ordinary integral element  by using claim
\ref{PropCpdeHCp} and the Cartan test.  \linebreak Finally,the Cartan-Kahler
theorem ensure then the existence of an  \linebreak  integral manifold and
lead us to conclude..

\end{enumerate}

\subsubsection*{Step 1:}

Let $(M,g)$ be an $m$-dimensional real analytic Riemannian manifold, where $g$ is a Riemannian metric, i.e. a 
covariant symmetric positive defined 2-tensor, such that at a given point
$p$ of $M$, $g_{p}$ in a orthonormed basis reduce to the identity matrix. However in a open neighborhood of
$p$, the matrix of $g$ can not always be the identity but it can always be  reduced to diagonal matrix:
\begin{equation}
g=g_{11}dx^{1}\otimes dx^{1}+g_{22}dx^{2}\otimes dx^{2}+ \dots +
g_{mm}dx^{m}\otimes dx^{m}
\end{equation}
where the  terms $g_{ii}$ are positive functions such that $g_{ii}=1$ at $p$.
We denote than $\eta_{i}=\sqrt{g_{ii}}dx^{i}$. $g$ can be written
as follows:

\begin{equation}
g=\eta_{1}\otimes \eta_{1}+ \eta_{2}\otimes \eta_{2}+ \dots +
\eta_{m}\otimes \eta_{m}
\end{equation}

$\eta=(\eta_{1},\eta_{2},\dots , \eta_{m})$ is than a orthonormal
coframe in the neighborhood of $p\in M$. We can establish the
Cartan's  structure equations:

\paragraph{\textbf{Cartan's structure equations on $M$:}}
\begin{equation}
\begin{cases}
\displaystyle{d\eta_{i}+\sum_{j=1}^{m}\eta_{ij}\wedge \eta_{j}=0 \qquad \text{( torsion-free )}}\\
\displaystyle{d\eta_{ij}+\sum_{k=1}^{m}\eta_{ik}\wedge
\eta_{kj}=\Omega_{ij}}
\end{cases}
\end{equation}
where $(\eta_{ij})$ is the matrix of  1-form of the
L\'evi-Civita's connection on $M$ (a torsion-free connection
compatible with the metric $g$).  $\Omega_{ij}$ is the
curvature 2-form of the connection.

\subsubsection*{Step 2:}

Let  $\mathbb{E}^{N}$ be an $N$-dimensional Euclidean space (for
the moment, $N>m$)  \linebreak endowed with the usual scalar product $\varepsilon_{N}$. Let us consider
$\mathcal{F}(\mathbb{E}^{N})$  a  positively-oriented orthonormal frame bundle
on $\mathbb{E}^{N}$. In what follows, we will not  work on the
entire bundle $\mathcal{F}(\mathbb{E}^{N})$,but on a quotient.  An element in $\mathcal{F}_{m}(\mathbb{E}^{N})$ has the form $(x;e_{1}, e_{2},\dots ,e_{m})$, where $x\in \mathbb{E}^{N}$ and $(e_{1},
e_{2},\dots ,e_{m})$ is a  positively-oriented orthonormal  set of vectors in
$\mathbb{E}^{N}$. We can consider
$\mathcal{F}_{m}(\mathbb{E}^{N})$ as follows: among all the positively-oriented
orthonormal frames of $\mathcal{F}(\mathbb{E}^{N})$, we take the
frames such that the first $m$ elements form a positively-oriented orthonormal set
of vectors, then we take  the $m$ first vectors of theses
frames. So,  $\mathcal{F}_{m}(\mathbb{E}^{N})$ is diffeomorphic to
 $\displaystyle{\mathbb{E}^{N}\times \frac{SO(N)}{SO(N-m)}}$.

\begin{prop}\label{DimFmEN}
\begin{equation}
dim \mathcal{F}_{m}(\mathbb{E}^{N})=N(m+1)-\frac{m(m+1)}{2}
\end{equation}
\end{prop}

 We define on  $\mathcal{F}(\mathbb{E}^{N})$ a set of  1-forms as follows\footnote{The indices  $i,j$ and
$k$ vary from 1 to $m$,
 the indexes  $a,b$ and $c$ vary from $m+1$ to $N$ and  the indexes $\mu , \nu$
 and  $\lambda$ vary from $1$ to $N$.}:

\begin{equation}
\omega_{\mu}=e_{\mu}dx \quad \text{ and }\quad
\omega_{\mu\nu}=e_{\mu}de_{\nu}=- e_{\nu}de_{\mu}=-\omega_{\nu\mu}
\end{equation}

So $(\omega_{1},\omega_{2},\dots ,\omega_{m},\omega_{m+1},\dots
,\omega_{N})$ form an orthonormal coframe of
$\mathcal{F}(\mathbb{E}^{N})$. Than the Cartan structure equations on $\mathcal{F}_{m}(\mathbb{E}^{N})$
are:

\begin{equation}
\begin{cases}
\displaystyle{d\omega_{\mu}+\sum_{\nu=1}^{N}\omega_{\mu\nu}\wedge \omega_{\nu}=0 \qquad \text{( torsion-free)}}\\
\displaystyle{d\omega_{\mu\nu}+\sum_{\lambda=1}^{N}\omega_{\mu\lambda}\wedge
\omega_{\lambda\nu}=0 \qquad \text{( flat curvature )}}
\end{cases}
\end{equation}

Notice that  $(\omega_{\mu\nu})$ is the $N\times N$ skew-symmetric
matrix connection form of the L\'evi-Civita  connection on $\mathbb{E}^{N}$.

\subsubsection*{Step 3:}

Let consider the product manifold $M\times
\mathcal{F}_{m}(\mathbb{E}^{N})$. Let $\mathcal{I}_{0}$ be the
exterior ideal on  $M\times \mathcal{F}_{m}(\mathbb{E}^{N})$
generated by the Paffafian system
$I_{0}=\{(\omega_{i}-\eta_{i}),\omega_{a}\}$.

\begin{prop}\label{ProCartanJanet}
Every $m$-dimensional integral manifold of  $\mathcal{I}_{0}$ on
which the form $\Delta=\omega_{1}\wedge\dots \wedge\omega_{m}$
does not vanish is locally the graph of a function
$f:M\longrightarrow \mathcal{F}_{m}(\mathbb{E}^{N})$ having the
property that $u=x\circ f :M \longrightarrow \mathbb{E}^{N}$ is a
local isometric embedding \footnote{ Conversely, each local
isometric embedding $u:M\longrightarrow\mathbb{E}^{N}$ come
uniquely from this construction.}.
\end{prop}

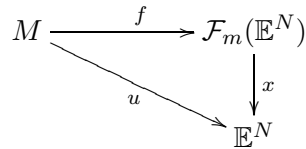
$$
\begin{xy}
\xymatrix{
      M \ar[rrd]_{u} \ar[rr]^{f}    & &    \mathcal{F}_{m}(\mathbb{E}^{N}) \ar[d]^{x}\\
      &   &\mathbb{E}^{N}}
\end{xy}
$$

\subsubsection*{Step 4:}

According to proposition \ref{ProCartanJanet}., the existence of
an integral manifold of $\mathcal{I}_{0}$ for wich
$\Delta$ is non zero, is a neccessary condition for the existence of a local
isometric embedding. However, the theorems and the results
that we discussed deal with closed exterior differential system . Therefore it is
natural to add to the Pffafian system $I_{0}$ the exterior
differentiation of each  1-form. We obtain so a closed exterior
differential system:
$I_{0}\cup dI_{0}$. When we compute the exterior differentiation of
$(\omega_{i}-\eta_{i})$, we remark new differential forms and an
interesting result,

\begin{equation}
d(\omega_{i}-\eta_{i})=-\sum_{j=1}^{m}(\omega_{ij}-\eta_{ij})\wedge\omega_{i}=0
\end{equation}

By Cartan's lemma,
$\displaystyle{\omega_{ij}-\eta_{ij}=\sum_{k=1}^{m}h_{ijk}\omega_{k}}$,
with $h_{ijk}=h_{ikj}=-h_{jik}$. With the
symmetry and the skew-symmetry of the functions $h_{ijk}$, we
conclude that $h_{ijk}$ are zero and so,
$\omega_{ij}-\eta_{ij}=0$. This result has a geometric
intrepretation: $\omega_{ij}-\eta_{ij}=0$ implies that
$f^{\ast}(\omega_{ij})=\eta_{ij}$ where $f$ is the function of
proposition. \ref{ProCartanJanet}, which means that the pull-back
of L\'evi-Civita connection by an isometric embedding is the
L\'evi-Civita connection on $M$.

So, we extend the exterior differential $I_{0}$ and we obtain  \linebreak an
exterior differential system on $M\times
\mathcal{F}_{m}(\mathbb{E}^{N})$
$I_{1}=\{(\omega_{i}- \linebreak \eta_{i})_{i=1,\dots , m},
(\omega_{a})_{a=m+1,\dots , N}, (\omega_{ij}-\eta_{ij})_{1\leq i <
j \leq m}\}$. In order to have a closed \linebreak  one, we add the exterior
differentiation of each form and we  get \linebreak  $I=I_{1}\cup dI_{1}$. We
denote $\mathcal{I}$ the exterior differential ideal generated by \linebreak 
$I=\{(\omega_{i}-\eta_{i}),\omega_{a},(\omega_{ij}-\eta_{ij}),
d(\omega_{i}-\eta_{i}),d\omega_{a},d(\omega_{ij}-\eta_{ij})\}$.

Instead of looking for integral manifold of $\mathcal{I}_{0}$, we
will look for the existence of an integral manifold of
$\mathcal{I}$.

From the structure equations stated earlier, we obtain the
following system:

\begin{equation}\label{Syst6equa}
\begin{cases}
d(\omega_{i}-\eta_{i})\equiv 0 \qquad mod \ I_{1} \\
\displaystyle{d\omega_{a}\equiv -\sum_{i=1}^{m}\omega_{ai}\wedge\omega_{i}} \qquad mod \ I_{1}\\
\displaystyle{d(\omega_{ij}-\eta_{ij})\equiv
\sum_{a=m+1}^{N}\omega_{ai}\wedge\omega_{aj}-\Omega_{ij}}\qquad
mod \ I_{1}
\end{cases}
\end{equation}

On $\mathbb{X}$, the integral manifold of $\mathbb{X}$,
$\omega_{a}=0$ , so $d\omega_{a}=0$ too. We conclude that
$\displaystyle{\sum_{i=1}^{m}\omega_{ai}\wedge\omega_{i}=0}$. The
Cartan lemma (lemma \ref{LemmedeCartan}., page
\pageref{LemmedeCartan}) ensures the existence of  $m^{2}$
functions $h_{aij}$ such that
$\displaystyle{\omega_{ai}=\sum_{j=1}^{m}h_{aij}\omega_{j}}$ where
$h_{aij}=h_{aji}$. We can write then:
$\displaystyle{\omega_{ai}-\sum_{j=1}^{m}h_{aij}\omega_{j}=0}$ on
$\mathbb{X}$.

However, nothing lead us to think that this equality will
be true outside $\mathbb{X}$. We define then  the
differential 1-form $\pi_{ai}$ on $M\times\mathcal{F}_{m}(\mathbb{E}^{N})$ as follows
\begin{equation}\label{PIai}
\pi_{ai}=\omega_{ai}-\sum_{j=1}^{m}h_{aij}\omega_{j}
\end{equation}

Consider now, the last equation of the system (\ref{Syst6equa})
\begin{equation}\label{EquSys6equaDerniere}
d(\omega_{ij}-\eta_{ij})\equiv
\sum_{a=m+1}^{N}\omega_{ai}\wedge\omega_{aj}-\Omega_{ij}\qquad
mod\ I
\end{equation}

On $\mathbb{X}$, $\omega_{ij}-\eta_{ij}=0$, so
$d(\omega_{ij}-\eta_{ij})=0$. restricted to $\mathbb{X}$,
(\ref{EquSys6equaDerniere}) becomes
\begin{equation}\label{Equa6SystSurX}
\sum_{a=m+1}^{N}\omega_{ai}\wedge\omega_{aj}=\Omega_{ij}.
\end{equation}

Using (\ref{PIai}), we can write
(\ref{Equa6SystSurX}) as follows
\begin{equation}\label{Omegaij=haikhajl}
\text{On }\mathbb{X}:\quad \Omega_{ij}=
\sum_{k,l=1}^{m}\Big(\sum_{a=m+1}^{N}(h_{aik}h_{ajl}-h_{ail}h_{ajk})\Big)\omega_{k}\otimes\omega_{l}
\end{equation}

from
$\Omega_{ij}=\displaystyle{\sum_{k,l=1}^{m}\mathcal{R}_{ijkl}\eta_{k}\otimes\eta_{l}=
\sum_{k,l=1}^{m}\mathcal{R}_{ijkl}}\omega_{k}\otimes\omega_{l}$,
we conclude that

\begin{equation}\label{EquaGauss}
\sum_{a=m+1}^{N}(h_{aik}h_{ajl}-h_{ail}h_{ajk})=\mathcal{R}_{ijkl}
\end{equation}

Equation  (\ref{EquaGauss}) is called the Gauss
equation. \\

 We see that the exterior differential system
$\tilde{I}=\{(\omega_{i}-\eta_{i}), \omega_{a},
(\omega_{ij}-\eta_{ij}), \pi_{ai}\}$ when the Gauss's equation is
satisfied, generates the  exterior differential ideal $\mathcal{I}$.
Actually, the 1-forms $(\omega_{i}-\eta_{i})$ and $\omega_{a}$ belong to
$I$ and  to $\tilde{I}$. The 1-forms
$(\omega_{ij}-\eta_{ij})=0$. This implies that
$d(\omega_{i}-\eta_{i})=0$. The  1-forms $\pi_{ai}=0$, so  \linebreak 
$d\omega_{a}=0$. From the Gauss equation, $d(\omega_{ij}-\eta_{ij})=0$. Looking for integral  \linebreak elements of
 $I$ is equivalent
to looking for integral elements of $\tilde{I}$ for which  \linebreak  the Gauss
equation is satisfied. We shall proceed this in the 
following steps. Moreover, $\tilde{I}$ contains less differential
1-form  than the exterior differential \linebreak  system
$I$.

\subsubsection*{ Step 5:}

The functions $h_{aij}$ are symmetric  in their two last indeces. If we
consider an  $(N-m)$-dimensional euclidean space $\mathcal{W}$, then  the matrix $(h_{aij})$ can be viewed as a symmetric  element of
$\mathbb{R}^{m}$($i,j=1,\dots , m$) taking value in $\mathcal{W}$,
i.e. $(h_{aij})\in \mathcal{W}\otimes S^{2}(\mathbb{R}^{m})$.
Notice that $\displaystyle{dim\mathcal{W}\otimes
S^{2}(\mathbb{R}^{m})=(N-m)\frac{m(m+1)}{2}}$.

\begin{prop}\label{ProDimKm} Let $\mathcal{K}_{m}$ the set of Riemannian curvature tensors $\mathcal{R}$ such that:
\begin{enumerate}
\item \label{SymR}$\mathcal{R}_{ijkl}=\mathcal{R}_{klij}$. \item
\label{AntiSymR} $\mathcal{R}_{ijkl}=-\mathcal{R}_{jikl}$. \item
\label{BianchiR}
$\mathcal{R}_{ijkl}+\mathcal{R}_{kijl}+\mathcal{R}_{jkil}=0$.
\end{enumerate}
where the indeces $i,j,k$ and $l$ vary from 1 to $m$.  Then
\begin{equation}
dim\mathcal{K}_{m}=\frac{m^{2}(m^{2}-1)}{12}
\end{equation}
\end{prop}

\begin{lemma}\label{LemmeSubSurjectThm}
Suppose that $\displaystyle{r=N-m\geq \frac{m(m-1)}{2}}$. Let $\mathcal{H}\subset
\mathcal{W}\otimes S^{2}(\mathbb{R}^{m})$ an open set containing the
elements $h=(h_{ij})$ such that the vectors  \linebreak $\{h_{ij} | 1\leq
i\leq j \leq m-1 \}$ are linearly independents as elements of
$\mathcal{W}$.  \linebreak The map
$\gamma:\mathcal{H}\longrightarrow\mathcal{K}_{m}$ that for $h\in
\mathcal{H}$ associate $\gamma(h)\in \mathcal{K}_{m}$ such that
$\displaystyle{\Big(\gamma(h)\Big)_{ijkl}=\sum_{a=m+1}^{N}h_{aik}h_{ajl}-h_{ail}h_{ajk}}$
, is  a surjective submersion.
\end{lemma}

\subsubsection*{Step 6: The existence of an $m$-dimensional ordinary integral element}

Let $\mathcal{I}$ the exterior ideal of
$M\times\mathcal{F}_{m}(\mathbb{E}^{N})$ generated by \linebreak 
$\displaystyle{s=N(m+1)-\frac{m(m+1)}{2}}$ 1-forms:
$$\{\underbrace{(\omega_{i}-\eta_{i})_{i=1,\dots ,
m}}_{m},\underbrace{(\omega_{a})_{a=m+1,\dots
N}}_{N-m},\underbrace{(\omega_{ij}-\eta_{ij})_{1\leq i< j \leq
m}}_{\frac{m(m-1)}{2}},\underbrace{(\pi_{ai})_{i=1,\dots , m,
a=m+1,\dots N}}_{(N-m)m}\}$$.

Let $\mathcal{Z}=\{(x,\Upsilon,h)\in
M\times\mathcal{F}_{m}(\mathbb{E}^{N})\times \mathcal{H} |
\gamma(h)=\mathcal{R}(x)\}$. $\mathcal{Z}$ is a submanifold (the
fiber of $\mathcal{R}$ by a  submersion. The surjectivity of
$\gamma$ ensure that $\mathcal{Z} \neq \emptyset$ ). So,
\begin{equation}
\begin{split}
dim\mathcal{Z}&=dim M + dim \mathcal{F}_{m}(\mathbb{E}^{N})+\dim
\mathcal{H} - \dim \mathcal{K}_{m}\\
&=\underbrace{m}_{dimM}+\underbrace{N(m+1)-\frac{m(m+1)}{2}}_{dim\mathcal{F}_{m}(\mathbb{E}^{N})}+\underbrace{(N-m)\frac{m(m+1)}{2}}_{dim\mathcal{H}}
-\underbrace{\frac{m^{2}(m^{2}-1)}{12}}_{dim\mathcal{K}_{m}}
\end{split}
\end{equation}

We define the map $\Phi:\mathcal{Z}\longrightarrow
\mathcal{V}_{m}(\mathcal{I},\Delta)$ that associate to
$(x,\Upsilon, h)$ the $m$-plane at $(x,\Upsilon)$ annihilated by
 the 1-forms that generate $\mathcal{I}$(the exterior \linebreak 
differential system $\tilde{I}$). The map $\Phi$ is an embedding and so
$\Phi(\mathcal{Z})$ is a  \linebreak submanifold of
$\mathcal{V}_{m}(\mathcal{I},\Delta)$. We will show that
$\Phi(\mathcal{Z})$ contains only ordinary integral elements. In the proof, 
we will use the proposition.\ref{PropCpdeHCp}.\\

Let  $(x,\Upsilon,h)\in \mathcal{Z}$ be a point. Let denote
$E=\Phi(x,\Upsilon,h)$ the integral  \linebreak element defined as follows:
$E=\{v\in
T_{(x,\Upsilon)}\Big(M\times\mathcal{F}_{m}(\mathbb{E}^{N})\Big) |
(\omega_{i}-\eta_{i})(v)=\omega_{a}(v)=(\omega_{ij}-\eta_{ij})(v)=\pi_{ai}(v)=0
\}$.\\

$E$ is an  $m$-dimensional integral element. As a matter of fact,  $s$  \linebreak is
the number of differential forms that generate the ideal
$\mathcal{I}$ and \linebreak 
$\displaystyle{dim\Big(M\times\mathcal{F}_{m}(\mathbb{E}^{N})\Big)-m=N(m+1)-\frac{m(m+1)}{2}=s}$.\\

We will apply word by word the proposition \ref{PropCpdeHCp}. Let
$\mathcal{I}$ the exterior ideal of $M\times
\mathcal{F}_{m}(\mathbb{E}^{N})$ defined above\footnote{$M\times
\mathcal{F}_{m}(\mathbb{E}^{N})$ play the role  of the
manifold  "$M$" in the proposition. \ref{PropCpdeHCp}.}. This
ideal does  not contain  $0$-forms. $E\in
\mathcal{V}_{m}(\mathcal{I})$ at  $(x,\Upsilon)\in
M\times \mathcal{F}_{m}(\mathbb{E}^{N})$. Let
$\omega_{i},
(\omega_{i}-\eta_{i}),\omega_{a},(\omega_{ij}-\eta_{ij}),\pi_{ai}$ be \linebreak  a coframe\footnote{There is $m+s
= dim\Big(M\times \mathcal{F}_{m}(\mathbb{E}^{N})\Big)$ 1-forms.}
of $M\times \mathcal{F}_{m}(\mathbb{E}^{N})$ in the neighborhood
of $(x,\Upsilon)$ such that\footnote{ the $(\omega_{i})_{i=1,\dots , m}$ play the role of "$\omega_{1},\omega_{2},\dots ,\omega_{n}$". the
$(\omega_{i}-\eta_{i}),\omega_{a},(\omega_{ij}-\eta_{ij}),\pi_{ai}$
play the role of  "$\pi_{s}$" in the  proposition
\ref{PropCpdeHCp}.} $E= \linebreak \{v\in
T_{x,\Upsilon}\Big(M\times\mathcal{F}_{m}(\mathbb{E}^{N})\Big) |
(\omega_{i}-\eta_{i})(v)=\omega_{a}(v)=(\omega_{ij}-\eta_{ij})(v)=\pi_{ai}(v)=0
\}$. \\

For $p\leq m$, we define the  $p$-dimensional integral element\footnote{the exterior differential system
 $I$ play the role of $\{\varphi_{1},\varphi_{2},\dots ,
\varphi_{r}\}$" in the  proposition \ref{PropCpdeHCp}.} \linebreak 
$E_{p}=\{x\in E | \omega_{k}(v)=0 \text{ pour tout } k
> p\}$ .\footnote{$E_{p}\in \mathcal{V}_{p}(\mathcal{I},\Delta)$
Because it is annihilated by  $s+m-p$ differential 1-forms .}. We
obtain so,  an integral flag $(0)_{(x,\Upsilon)}=E_{0}\subset
E_{1}\subset E_{2}\subset \dots \subset E_{m}\subset
T_{(x,\Upsilon)}\Big(M\times\mathcal{F}_{m}(\mathbb{E}^{N})\Big)$.
We remind that $I=\{\underbrace{(\omega_{i}-\eta_{i}),
\omega_{a},(\omega_{ij}-\eta_{ij})}_{\text{differential } 1-\text{forms}}, \underbrace{d(\omega_{i}-\eta_{i}),
d\omega_{a},d(\omega_{ij}-\eta_{ij})}_{\text{differential } 2-\text{forms}}\}$..\\

By computing  $d(\omega_{i}-\eta_{j}),d\omega_{a}$ and
$d(\omega_{ij}-\eta_{ij})$, we shall find the differential forms
that are linear combinations of the forms which generate
$\mathcal{I}$.\footnote{the forms that play the role of
$\pi_{\rho}^{J}$ in the proposition \ref{PropCpdeHCp}.}

After simple calculations, we find that

\begin{equation}\label{1formei}
d\omega_{a}\equiv -\sum_{i=1}^{m}\pi_{ai}\wedge\omega_{i}
\end{equation}

and

\begin{equation}
d(\omega_{ij}-\eta_{ij})=\underbrace{\sum_{a=m+1}^{N}\pi_{ai}\wedge\pi_{aj}}_{\blacklozenge}+
\sum_{k=1}^{m}\Big(\sum_{a=m+1}^{N}h_{ajk}\pi_{ai}-h_{aik}\pi_{aj}\Big)\wedge\omega_{k}
\end{equation}

the term $(\blacklozenge)$ is  quadratic in $\pi_{ai}$ and vanishes
on $\mathbb{X}$.\footnote{($\blacklozenge$) play the role of
$\tilde{\varphi}_{\rho}$ in the proposition \ref{PropCpdeHCp}.}\\

According to proposition \ref{PropCpdeHCp}., $c_{p}$ represente
the number of linear \linebreak independent  differential  1-forms.\\

\begin{tabular}{|c|c|c|}
\hline
The differential 1-forms  & The indexes & Number of linear \\
 & & independent 1-forms\\
\hline
$\omega_{i}-\eta_{i}$ & $1\leq i\leq m$ & $m$\\
\hline
$\omega_{a}$ & $m+1\leq a\leq N$ & $N-m$\\
\hline
$\omega_{ij}-\eta_{ij}$ & $1\leq i < j\leq m$ & $\displaystyle{\frac{m(m-1)}{2}}$\\
\hline
$\pi_{ai}$  & $1 \leq i\leq p$,& $(N-m)p$\\
 & $m+1\leq a\leq N$ &\\
\hline
$\displaystyle{\sum_{a=m+1}^{N}(h_{aik}\pi_{aj}-h_{ajk}\pi_{ai})}$
&
$1\leq k \leq p$,& $\displaystyle{p\frac{(m-p)(m-p-1)}{2}+}$\\
 &  $1\leq i \leq j \leq m$ & $\displaystyle{\frac{p(p+1)}{2}(m-p)}$\\
\hline
\end{tabular}

\paragraph{}Finally, by the sum of the number of linear independent 1-forms of the above table,
 $c_{p}$ is the  codimension of
$H(E_{p})$  in $G_{m}\Big(T(M\times
\mathcal{F}_{m}\mathbb{E}^{N})\Big)$ defined earlier, and is equal
to:

\begin{equation}
c_{p}=N+\frac{m(m-1)}{2}+(N-m)p+\frac{mp(m-p)}{2}
\end{equation}

so,

\begin{equation}
\sum_{p=0}^{m-1}c_{p}=\frac{Nm(m+1)}{2}+\frac{m^{2}(m^{2}-1)}{12}.
\end{equation}

To apply the proposition \ref{PropCpdeHCp} and show that  $E$ is
an $m$-dimensional  \linebreak ordinary integral element, we need just to
compute the codimension of $\Phi(\mathcal{Z})$  \linebreak in
$G_{m}(\mathcal{I},\Delta)$. \\

Let  $\mathfrak{U}$ an open set of
$\mathcal{F}_{m}(\mathbb{E}^{N})$. So,  $dim
(M\times\mathfrak{U})=\displaystyle{N(m-1)-\frac{m(m+1)}{2}+m}.$
We remid that if $E$ is an $n$-dimensional euclidean space, then
the space of all  $p$-planes of $E$ ($G_{p}(E)$), with $p<n$, is
of  $p(n-p)$ dimension. Let $(x,\Upsilon) \in
M\times\mathfrak{U}$,

\begin{equation}\nonumber
\begin{split}
&dim G_{m}\Big(T(M\times \mathfrak{U})\Big)=dim
G_{m}\Big(T_{(x,\Upsilon)}(M\times\mathfrak{U})\Big)+ dim
(M\times\mathfrak{U})\\
&=m\Big(\underbrace{N(m+1)-\frac{m(m+1)}{2}+m}_{\tiny{dim
T_{(x,\Upsilon)}(M\times\mathfrak{U})=dim (M\times
\mathfrak{U})}}-m\Big)+ \underbrace{N(m+1)-\frac{m(m+1)}{2}+m}_{\tiny{dim
(M\times \mathfrak{U})}}\\
&=m\Big(N(m+1)-\frac{m(m+1)}{2}\Big)+ N(m+1)-\frac{m(m+1)}{2}+m
\end{split}
\end{equation}

Since that $\Phi$ is an embedding,
$dim\Phi(\mathcal{Z})=dim\mathcal{Z}$, so

\begin{equation}\nonumber
\begin{split}
& dim G_{m}\Big(T(M\times \mathfrak{U})\Big)-dim
\Phi(\mathcal{Z})=\\
&\underbrace{Nm(m+1)-\frac{m^{2}(m+1)}{2}+N(m+1)-\frac{m(m+1)}{2}+m}_{\tiny{dimG_{m}\Big(T(M\times
\mathfrak{U})\Big)}}
\\&\ \ \ \underbrace{-m-N(m+1)+\frac{m(m+1)}{2}
-\frac{Nm(m+1)}{2}+\frac{m^{2}(m+1)}{2}+\frac{m^{2}(m^{2}-1)}{12}}_{\tiny{dim
\Phi(\mathcal{Z})}}\\
&=\frac{Nm(m+1)}{2}+\frac{m^{2}(m^{2}-1)}{12}
\end{split}
\end{equation}
 We conclude that the  codimension of $\Phi(\mathcal{Z})$ in
$G_{m}\Big(T(M\times \mathfrak{U})\Big)$ is equal to
$c_{0}+c_{1}+\dots +c_{m-1}$. By the Cartan's test,
$E\in\mathcal{V}_{m}(\mathcal{I},\Omega)$ is an ordinary integral
element of $\mathcal{I}$. The Cartan-K\"ahler theorem (Corollary
\ref{CorCartanKahler}.) ensure the existence of an integral
manifold $\mathbb{X}$ passing through $(x,\Upsilon)$ and having
$E$ as a tangent space at $(x,\Upsilon)$.\\
 $E\in \mathcal{V}_{m}(\mathcal{I},\Omega)$, In particular, $E\in
 \mathcal{V}_{m}(\mathcal{I}_{0},\Omega)$. By the proposition
 \ref{ProCartanJanet}. , there exists an isometric embedding of
 $(M,g)$ in $(\mathbb{E}^{N},\varepsilon_{N})$.

\section{Local Conformal Embedding Problem}
\begin{definition}
Let $(M,g)$ and $(N,h)$ be two real analytic Riemannian  \linebreak manifolds of
dimension $m$ and $n$ respectively. Let $f$ be  a map from  $M$ to
$N$.  \linebreak Then $f$ is a conformal embedding from $(M,g)$ to $(N,h)$ if:
\begin{enumerate}
\item $f$ is a local diffeomorphism; \item $f^{\ast}h=Sg$, where
$S:M\longrightarrow \mathbb{R}^{+}$ is a stricly positive function
on $M$.
\end{enumerate}
\end{definition}

\begin{theorem}(Jacobowitch-Moore \cite{JacobowitchMoore-Art})\\
If $dimN=n \geq \frac{1}{2}m(m+1)-1$, then each  point $p \in M$
admit a neighborhood on $M$ which can be conformally embedded in  $N$.
\end{theorem}

Jacobowitch and Moore gave two differents  proofs of  this result; one is based on Janet's method and the second on Cartan's method which is close to the proof of Burstin-Cartan-Janet-Schlaefli theorem that we gave.

Roughly speaking, we consider $\mathcal{F}(M)\times \mathcal{F}(N)\times\mathbb{R}^{m}\times\mathbb{R}^{+}_{*}$ and we look for integral manifolds of $I_{0}=\{\omega_{i}- S\eta_{i}, \omega_{a}\}$ (the forms are defined as on the previous section). Similarly, we  can extend the exterior differential system to obtain a closed one. We lead the deatils of the proof for the reader who should take care of $S$ when he apply the exterior differentiation cause it's a function. When the new exterior differential system is involutive, we look for ordinary integral element and so conclude. (see \cite{JacobowitchMoore-Art}).

\bibliographystyle{plain}

\end{document}